\documentclass[a4paper,10pt]{article}
\usepackage[utf8]{inputenc}
\usepackage[english]{babel}
\usepackage{amssymb, amsmath, enumerate}
\usepackage[all,2cell]{xy}
\usepackage{hyperref}
\usepackage{bm}

\usepackage{amsthm}
\newtheorem{teo}{Theorem}[section]
\newtheorem{lem}[teo]{Lemma} 

\newtheorem{defn}[teo]{Definition} 
\newtheorem{ex}[teo]{Example}

\newtheorem{fat}[teo]{Fact}
\newtheorem{rem}[teo]{Remark}

\author{
    Kaique Matias de Andrade Roberto\\
    \texttt{kaique.roberto@usp.br}
    \and
    Hugo Luiz Mariano\\
    \texttt{hugomar@ime.usp.br} \\
    Instituto de Matematica e Estatistica \\
    Universidade de Sao Paulo, Brazil
}
\title{On superrings of polynomials and algebraically closed multifields
}

\usepackage{filecontents}

\date{}
\begin{document}
    
    \maketitle
    
    \begin{abstract}
        
        The concept of multialgebraic structure -- an ``algebraic like'' structure but endowed with  multiple valued operations -- has 
        been studied since the 1930's; in particular, the concept of hyperrings was introduced by Krasner in the 1950's. Some general 
        algebraic study has been made on multialgebras: see for instance \cite{golzio2018brief} and \cite{pelea2006multialgebras}. More 
        recently the notion of multiring have obtained more attention: a multiring is a lax hyperring, satisfying an weak distributive 
        law, but  hyperfields and multifields coincide. Multirings   has been studied for  applications in abstract quadratic forms theory 
        (\cite{marshall2006real}, \cite{worytkiewiczwitt2020witt}) and tropical geometry (\cite{jun2015algebraic}); a more detailed account of 
        variants of concept of polynomials over hyperrings is even more recent (\cite{jun2015algebraic}, \cite{ameri2019superring}). In 
        the present work we start a model-theoretic oriented analysis of multialgebras  introducing the class of algebraically closed  and 
        providing variant proof of quantifier elimination flavor,   based on new results on  superring of polynomials (\cite{ameri2019superring}).

        \vspace{0.3cm}
        
        
        
        
    \end{abstract}
    

    \section*{Introduction}
    \hspace*{\parindent}
    

    The concept of multialgebraic structure -- an ``algebraic like'' structure but endowed with  multiple valued operations -- has 
    been studied since the 1930's; in particular, the concept of hyperrings was introduced by Krasner in the 1950's. 
    
    Some general algebraic study has been made on multialgebras: see for instance \cite{golzio2018brief} and 
    \cite{pelea2006multialgebras}. 
    
    More recently the notion of multiring have obtained more attention: a multiring is a lax hyperring, satisfying an weak 
    distributive law, but  hyperfields and multifields coincide. Multirings   has been studied for  applications in abstract 
    quadratic forms theory (\cite{marshall2006real}, \cite{worytkiewiczwitt2020witt}) and tropical geometry (\cite{jun2015algebraic}); a more 
    detailed account of variants of concept of polynomials over hyperrings is even more recent (\cite{jun2015algebraic}, 
    \cite{ameri2019superring}).

    In the present work we start a model-theoretic oriented analysis of multialgebras  introducing the class of algebraically closed  
    and providing variant proof of quantifier elimination flavor,   based on new results on  superring of polynomials as an Euclidean 
    algorithm of division.
    
    Overview of the work. In section 1 we develop the preliminaries results needed for the paper. Section 2 is devoted to a detailed account of the construction of supperring of 
    polynomials  (\cite{ameri2019superring}) and to present some new results as the Euclidean algorithm of division for the superring of polynomials 
    with coefficients over a hyperfield. Section 3 contains 
    the main contributions of this paper: we  introduce some concept of  the algebraically closed hyperfield and give the first steps on a model theory of this class with a kind of quantifier elimination procedure. We finish the work in section 4 presenting some possible future developments.
    
    \section{Preliminaries}
    
    Our goals in this section are to develop the preliminary results needed for the work, and to provide a brief dictionary on 
    multialgebras and hyperrings. We split it in two subsections, the first one contains general definitions and results on  
    multialgebraic structures and the second one is focused specially on rings-like multi structures.
    
    \subsection{On Multialgebras}
    
    There are several definitions of multialgebra in the literature, considering that each multialgebra application in a specific area 
    of Mathematics (mainly Algebra and  Logic) requires a particular adaptation. Here, we adapt the notion of multialgebra used in \cite{coniglio2017non}; the identity theory here presented is close to the exposed in  \cite{pelea2006multialgebras}. 
    
    \begin{defn}
        A \textbf{multialgebraic signature} is a sequence of parwise disjoint sets 
        $$\Sigma=(\Sigma_n)_{n\in \mathbb{N}},$$ 
        where $\Sigma_n=S_n\sqcup M_n$, which $S_n$ is the set of strict multi-operation symbols and $M_n$ is 
        the set of multioperation symbols. In particular,  $\Sigma_0=S_0 \sqcup M_0$, $F_0$ 
        is the set of symbols for constants and $M_0$ is the set of symbols for multi-constants. We also denote 
        $$\Sigma=((S_n)_{n\ge0},(M_n)_{n\ge0}).$$
    \end{defn}
    
    \begin{defn} Let $A$ be any set.
        \begin{enumerate}[i -]
            
            \item A multi-operation of arity $n \in \mathbb{N}$ over a set $A$ is a function 
            $$A^n \to \mathcal P^*(A):=\mathcal P(A)\setminus\{\emptyset\} .$$ 
            
            \item A multi-operation of arity $n \in \mathbb{N}$ over a set $A$, $A^n \to \mathcal P^*(A)$, is {\em strict}, 
            whenever it factors throuth the singleton function $s_A : A \rightarrowtail \mathcal P^*(A)$, $a\mapsto s_A(a) := \{ a\}$. Thus 
            it 
            can be naturally identified with an ordinary n-ary operation  $A^n \to A$.
        \end{enumerate}
    \end{defn}
    A $0$-ary multi-operation (respectively {\em strict} multi-operation)  on $A$ can be identified with a non-empty subset of $A$ 
    (respectively a singleton subset of $A$).
    
    \begin{defn}
        A \textbf{multialgebra} over a signature $\Sigma=((S_n)_{n\ge0},(M_n)_{n\ge0})$, is a set $A$ endowed with a family of n-ary 
        multioperations 
        $$\sigma^A_n : A^n \to \mathcal P^*(A),\, \sigma_n \in S_n \sqcup M_n,\, n \in \mathbb N,$$  
        such that: if $\sigma_n \in S_n$, then $\sigma^A_n : A^n \to \mathcal P^*(A)$ is a {\em strict} n-ary multioperation.
    \end{defn}
    
    \begin{rem}
        $ $
        \begin{enumerate}[i -]
            \item Every algebraic signature $\Sigma = (F_n)_{n \in \mathbb N}$ is  a multialgebraic signature where $M_n = \emptyset, \forall 
            n \in \mathbb N$. Each algebra 
            $$(A, ((A^n \overset{f^A}\to A)_{f \in F_n})_{n \in \mathbb N})$$ 
            over the algebraic signature $\Sigma$ can be naturally identified with a multi-algebra 
            $$(A, ((A^n \overset{f^A}\to A  \overset{s_A}\rightarrowtail \mathcal P^*(A))_{f \in F_n})_{n \in \mathbb N})$$ 
            over the same signature.
            
            \item Every multialgebraic signature $\Sigma = ( (S_n)_{n \in \mathbb N}, (M_n)_{n \in \mathbb N})$ induces naturally a 
            first-order language 
            $$L(\Sigma) = ((F_n)_{n \in \mathbb N}, (R_{n+1})_{n \in \mathbb N})$$ 
            where $F_n := S_n$ is the set of n-ary operation symbols and $R_{n+1} := M_n$ is the set of (n+1)-ary relation symbols. In this 
            way, 
            multi-algebras 
            $$(A, ((A^n \overset{\sigma^A}\to  \mathcal P^*(A))_{\sigma  \in S_n \sqcup M_n})_{n \in \mathbb N})$$ 
            over a multialgebraic signature $\Sigma = (S_n \sqcup M_n)_{n \in \mathbb N}$ can be naturally identified with the first-order 
            structures over the language  $L(\Sigma)$ that satisfies the $L(\Sigma)$-sentences:
            $$\forall x_0 \cdots \forall x_{n-1} \exists x_n (\sigma_n(x_0, \cdots, x_{n-1}, x_n)), \ \text{for each}\ \sigma_n \in R_{n+1} = 
            M_n, n \in \mathbb N.$$
        \end{enumerate}
    \end{rem}
    
    Now we focus our attention into a more syntactic aspect of this multi-algebras theory. We start with a (recursive) definition of 
    multi-terms:
    
    \begin{defn}
        A \textbf{\em (multi-)term} on a multialgebra $A$ of signature 
        $$\Sigma=((S_n)_{n\ge0},(M_n)_{n\ge0})$$
        is defined recursively as:
        \begin{enumerate}[i -]
            \item Variables $x_i, i \in \mathbb N$ are terms.
            \item If $t_0,\cdots,t_{n-1}$ are terms and $\sigma \in S_n \sqcup M_n$, then $\sigma(t_0,\cdots, t_{n-1})$ is a term.
        \end{enumerate}
        We will call a multi-term $t$ \textbf{\em strict}, whenever it is composed only by combination of \textbf{\em strict} 
        multi-operations and variables. The notion of \textbf{\em occurrence} of a  variable in a term is as the usual. We will denote 
        $var(t)$ as the (finite set of variables) that occurs in the term $t$.
    \end{defn}
    
    To define an interpretation for terms, we need a preliminary step. Given 
    $$\sigma \in S_n \sqcup M_n,$$
    we ``extend'' $\sigma^A : A^n \to \mathcal P^*(A)$ to a n-ary operation in $\mathcal P^*(A)$, 
    $$\sigma^{\mathcal P^*(A)} : \mathcal P^*(A)^n \to \mathcal P^*(A),$$
    by the rule:
    $$\sigma^{\mathcal P^*(A)}(A_0, \cdots ,A_{n-1}) :=
    \bigcup\limits_{a_0\in A_0} \cdots \bigcup\limits_{a_{n-1}\in A_{n-1}}\sigma^A(a_0,\cdots,a_{n-1}).$$
    
    \begin{defn}
        The \textbf{interpretation of a term} $t$ on a multialgebra $A$ over a signature $\Sigma=((S_n)_{n\ge 0},(M_n)_{n\ge 0})$ is a 
        function $t^A : A^{var(t)} \to \mathcal P^*(A)$ and  is defined 
        recursively as follows:
        \begin{enumerate}[i -]
            \item The interpretation of a variable $x_i$, $x_i^A:A^{\{x_i\}}\rightarrow\mathcal P^*(A)$ is essentialy the singleton function 
            of $A$: 
            $$x_i^A:A^{\{x_i\}}\cong A\rightarrowtail\mathcal P^*(A),\,\mbox{is given by the rule }
            (\hat a:\{x_i\}\rightarrow A)\mapsto \{a\}.$$
            
            \item If $t = \sigma(t_0,\cdots,t_{n-1})$ is a term and $\sigma\in S_n\sqcup M_n$, denote $T=var(t)$ and $T_i=var(t_i)$. Then 
            $T=\bigcup_{i <n} T_i$. Consider ${t_i}^A_T : A^T \to \mathcal P^*(A)$ the composition 
            $$A^T \overset{proj^T_{T_i}}\twoheadrightarrow A^{T_i} \overset{t_i^A}\to \mathcal P^*(A),$$
            where $proj^T_{T_i}$ is the canonical projection induced by the inclusion $T_i \hookrightarrow T$. Then $t^A : A^T \to \mathcal 
            P^*(A)$ is the composition 
            $$A^T \overset{({t_i}^A_T)_{i <n}}\longrightarrow (\mathcal P^*(A))^n \overset{\sigma^{\mathcal P^*(A)}}\longrightarrow 
            \mathcal P^*(A).$$
        \end{enumerate}
    \end{defn}
    
    \begin{defn} Let $A$ be a multialgebra $A$ over a signature $\Sigma=((S_n)_{n\ge 0},(M_n)_{n\ge 0})$ and let $t_1,t_2$ be 
        $\Sigma$-terms. We say that $A$ realize that $t_1$ is \textbf{\em contained in} $t_2$, 
        (notation: $A \models t_1\sqsubseteq t_2$) whenever $t_1^A(\bar{a})\subseteq t_2^A(\bar{a})$, for each tuple $\bar{a} : var(t_1) 
        \cup var(t_2) \to A$. 
    \end{defn}
    
    Apart from the notion of atomic formulas the definition of $\Sigma$-formulas for multi-algebraic theories is similar to the 
    (recursive) definition  of  first-order $L(\Sigma)$-formulas:
    
    \begin{defn} The formulas of $\Sigma$ are defined as follows:
        \begin{enumerate}[i-]
            \item  Atomic formulas are the formulas of type $t\sqsubseteq t'$, where  $t,t'$ are terms.
            \item If $\phi, \psi$ are formulas, then $\neg\phi$ and  $\phi \vee \psi, \phi \wedge \psi, \phi \to \psi, \phi \leftrightarrow 
            \psi$ are formulas.
            \item If $\phi$ is a formula and $x_i$ is a variable, then $\forall x_i \phi$, $\exists x_i \phi$ are formulas.
        \end{enumerate}
        The notion of {\em occurrence} (respec. free occurrence) of a  variable in a formula is as the usual. We will denote $fv(\phi)$ 
        as the (finite) set of variables that occurs free in the formula $\phi$.

        We use $t_1=_st_2$ to abbreviate the formula $(t_1\sqsubseteq t_2) \wedge (t_2\sqsubseteq t_1)$: this means that $t_1$ and $t_2$ 
        are "strongly equal terms".
    \end{defn}
    
    \begin{defn}
        The definition of interpretation of formulas $\phi (\bar{x})$ where 
        $$fv(\phi) \subseteq \bar{x} \subseteq  \{ x_i: i \in \mathbb N\}$$ 
        under a valuation of variables $v : \bar{x} \to A $ (or we will denote simply by  $v = \bar{a}$) is:
        \begin{enumerate}[i-]
            \item $A \models_{v} t (\bar{x}) \sqsubseteq t'(\bar{x})$ iff $t^A(\bar{a}) \subseteq t'^A(\bar{a})$
            \item  The case of complex formulas (given by the connectives $\neg$, $\vee$, $\wedge$, $\to$, $\leftrightarrow$, and quantifiers 
            $\forall, \exists$) is as satisfaction of first-order $L(\Sigma)$-formulas in $L(\Sigma)$-structure on a valuation $v$.
        \end{enumerate}
    \end{defn}
    
    \begin{rem}
        $ $
        \begin{enumerate}[i-]
            \item  The theory of multi-algebras entails that for each term $t$, and each {\em strict} term $t'$,
            $$ t \sqsubseteq t' \mbox{ iff  } t =_s t'.$$
            
            \item In \textup{\cite{pelea2006multialgebras}}  contains a development of the identity theory for multialgebras, with another primitive 
            notion: $t(\bar{x})=_w t'(\bar{x})$; a $\Sigma$-multialgebra $A$ satisfies the "weak identity" above   iff there is some $\bar{a} \in 
            A^{var(t) \cup var(t')}$ such that $t^A(\bar{a}) \cap {t'}^A(\bar{a}) \neq \emptyset$. This will not play any role in this work but is useful 
            for applications of multi-algebraic semantics for complex logical systems (\textup{\cite{golzio2018brief}}).
        \end{enumerate}
    \end{rem}
    
    There are many ways of define morphism for multialgebras. Follow below our choice:
    
    \begin{defn}
        Let $A$ and $B$ be multialgebras of signature $\Sigma=((S_n)_{n\ge 0},(M_n)_{n\ge 0})$ and $\varphi:A\rightarrow B$ be a 
        function.
        \begin{enumerate}[i -]
            \item $\varphi$ is a \textbf{partial morphism} if for every $n\ge0$, every $\sigma\in S_n$ and every $a_1,...,a_n\in A$, we have
            $$\varphi(\sigma^A(a_1,...,a_n))\subseteq\sigma^B(\varphi(a_1),...,\varphi(a_n)).$$
            
            \item $\varphi$ is a \textbf{morphism} if for every $n\ge0$, every $\sigma\in S_n\sqcup M_n$ and every $a_1,...,a_n\in A$, we 
            have
            $$\varphi(\sigma^A(a_1,...,a_n))\subseteq\sigma^B(\varphi(a_1),...,\varphi(a_n)).$$
            
            \item $\varphi$ is a \textbf{strong morphism} if for every $n\ge0$, every $\sigma\in S_n\sqcup M_n$ and every $a_1,...,a_n\in 
            A$, we have
            $$\varphi(\sigma^A(a_1,...,a_n))= \sigma^B(\varphi(a_1),...,\varphi(a_n)).$$
        \end{enumerate}
    \end{defn}
    
    \begin{rem}\label{translation-rem}
        $ $
        \begin{enumerate}[i -]
            \item Let $A, B$ be  $\Sigma$-multialgebras. If $B$ is a {\em strict multilagebra} (i.e. $\sigma_n^B(\bar{b})$ is unitary subset 
            of $B$, for each $\sigma \in \Sigma$ and each tuple $\bar{b}$ in $B$),  then the morphisms $A \to B$ coincide with
            the strong morphisms $A \to B$.
            
            \item There is a full and faithful concrete embedding of the category of ordinary algebraic structures  over a signature $\Sigma$ 
            and homomorphisms into the category of $\Sigma$-multialgebras and (strong) morphisms: the image of this embedding is the class of 
            strict multialgebras over $\Sigma$.
            
            \item  The correspondence $\Sigma \mapsto  L(\Sigma)$ induces a {\em concrete} isomorphism between the category of 
            $\Sigma$-multialgebras and the category of $L(\Sigma)$- first order structures satisfying suitable $\forall \exists$ axioms. It is 
            ease to see that  this correspondence  induces a bijection between injective strong embeddings of $\Sigma$-multialgebtras and 
            $L(\Sigma)$-monomorphisms of first-order structures.
        \end{enumerate}
    \end{rem}
    
    We finish this subsection with two illustrative examples of multialgebras derived from an algebraic structure and from a 
    first-order structure. 
    
    \begin{ex}
        Let $(R,+,\cdot,0,1)$ be a commutative ring with $1\ne0$. Given $n\ge1$, define an $(n+1)$-ary multioperation $\ast_n$ by the 
        rule:
        \begin{align*}
            d\in a_0\ast_na_1\ast_na_2\ast_n...\ast_na_n&\Leftrightarrow\mbox{ there is some }t\in R\mbox{ such that } \\
            &d=a_0+a_1t+a_2t^2+...+a_nt^n.
        \end{align*}
        
        The idea here, is that $a_0\ast_na_1\ast_na_2\ast_n...\ast_na_n$ ``analyze'' the values taken in $R$ by the 
        polynomial $p(X)=a_0+a_1X+a_2X^2+...+a_nX^n\in R[X]$. $\ast_n$ will be called \textbf{The streching multialgebra of degree $n$} 
        over $R$.
    \end{ex}
    
    \begin{ex}\label{ordermulti}
        Let $\mathcal L=\{0,1,+,\cdot,\le\}$ the language of ordered fields. Consider $\mathbb R$ as an ordered field. We can look at 
        the ordering relation as a multioperation of arity 1. In agreement with our notation, we have
        \begin{align*}
            \le(a):=\{x\in\mathbb R:a\le x\}=[a,+\infty).
        \end{align*}
    \end{ex}
    
    $ $
    
    From now on, all multi-algebras considered in this work will contain only operations of arities $0,1,2$. They will have strict constants and strict unary operations; the binary operations maybe strict or multivalued.
    
    \subsection{Multirings and Superrings}
    
    Now, we will get closer to the subject of our work.
    
    \begin{defn}[Adapted from definition 1.1 in \textup{\cite{marshall2006real}}]\label{defn:multimonoid}
        Let $\Sigma = ((S_n)_{n\ge0},(M_n)_{n\ge0})$ be a multialgebraic signature where $S_{0} = \{1\}$, $S_{1} = \{r\}$, 
        $S_n =\emptyset$ for all  $n \neq 0,1$ and $M_2 = \{\cdot\}$, $M_n = \emptyset$ for all $n \neq 2$ . A multigroup is a 
        $\Sigma$-structure  $(G,\cdot,r,1)$ where $G$ is a (non-empty) set, $1$ is an element of 
        $G$, $r:G\rightarrow G$ is a function,  $\cdot:G\times G\rightarrow \mathcal P^*(G)$,  that satisfies the following formulas:
        \begin{enumerate}[i -]
            \item $G\models 1\cdot x=_sx$.
            \item $G\models x\cdot 1=_sx$.
            \item $G\models [(x\cdot y)\cdot z] =_s [x\cdot (y\cdot z)]$.
            \item $G\models(z\sqsubseteq x\cdot y)\rightarrow[(x\sqsubseteq z+\cdot r(y))\wedge(y\sqsubseteq r(x)\cdot z)]$.
        \end{enumerate}
        
        A multimonoid is a multialgebra such that $S_1 = \{1\}$, $M_2 = \{\cdot\}$ and the other sets of  symbols  are empty, that 
        satifies axioms (i), (ii), (iii) above.
        
        A multimonoid/multigroup will said to be \textbf{commutative (or abelian)} if satisfy: 
        
        $$G\models x\cdot y  =_s  y \cdot x.$$ 
        
        
        
        For multigroups,  axiom (iii) can be replaced by the (apparently weaker) version: 
        $$G\models [(x\cdot y)\cdot z] \sqsubseteq [x\cdot (y\cdot z)]$$
    \end{defn}
    
    In other words, an abelian multigroup is a first-order structure  $(G,\cdot,r,1)$ where $G$ is a non-empty set, $r:G\rightarrow G$ 
    is a function, $1$ is an 
    element of $G$, $\cdot \subseteq G\times G\times G$ is a ternary relation (that will play the role of binary multioperation, we 
    denote $d\in a\cdot 
    b$ for $(a,b,d)\in\cdot$) such that for all $a,b,c,d\in G$:
    \begin{description}
        \item [M1 - ] If $c\in a\cdot b$ then $a\in c\cdot(r(b))\wedge b\in(r(a))\cdot c$. We write $a-b$ to simplify $a+(-b)$.
        \item [M2 - ] $b\in a\cdot1$ iff $a=b$.
        \item [M3 - ] If $\exists\,x(x\in a\cdot b\wedge t\in x\cdot c)$ then
        $\exists\,y(y\in b\cdot c\wedge t\in a\cdot y)$.
        \item [M4 - ] $c\in a\cdot b$ iff $c\in b\cdot a$.
    \end{description}
    
    \begin{ex}
        $ $
        \begin{enumerate}[a-]
            \item  Suppose that $(G,\cdot, ( \ )^{-1}, 1)$ is a ordinary group. Defining $a \ast b = \{a \cdot b\}$ and $r(g)=g^{-1}$, we have 
            that $(G,\ast,r,1)$ 
            is a multigroup. 
            
            \item (\textup{\cite{pelea2006multialgebras}}) Let $(G,\cdot, ( \ )^{-1}, e)$ be an ordinary group and let $S \subseteq G$ be a subset such that $e \in S$ and $S^{-1} 
            \subseteq S$, define a binary relation $a \sim_S b$ iff $b\cdot a^{-1} \in S$ . This  is a reflexive and symmetric relation. Then 
            take  $\sim_S^t$ be the transitive closure of $\sim_S$ (note that if $S$ is a subgroup of $G$, then $\sim_S = \sim_S^t$). Then  
            $G/\sim^t_S$ with the inherit structure is a multigroup. In particular if $G$ is a commutative group and $S$ is a subgroup of 
            $G$, then $G/\sim^t_S$ with the inherit structure is an ordinary abelian group.
        \end{enumerate}
    \end{ex}
    
    \begin{defn}[Adapted from definition 2.1 in \textup{\cite{marshall2006real}}]\label{defn:multiring}
        A (commutative, unital) multiring is a multialgebraic structure $(R,+,\cdot,-,0,1)$ where $(R, +, - ,0)$ is a commutative 
        multigroup, $(R,\cdot, 1)$ is a  commutative (strict) monoid  and that also satisfies the following axioms:
        \begin{itemize}
            \item $R\models [x\cdot 0]=_s 0$ (zero is absorving). 
            \item $R\models[z.(x+y)]\sqsubseteq[z.x+z.y]$ (weak or semi  distributive law).
        \end{itemize}
        A multidomain is a non-trivial multiring without zero-divisors and a multifield is a non-trivial  multiring such that every 
        nonzero element is 
        invertible. 
        
        A multiring is an hyperring if it satifies the full distributive law: 
        $$R\models[z(x+y)]=_s[zx+zy].$$
        Of course, we extend this  terminology for hyperdomains and hyperfields.
    \end{defn}
    
    In other words, a multiring is a tuple $(R,+,\cdot,-,0,1)$ where $R$ is a non-empty set, $\cdot:R\times R\rightarrow R$
    and $-:R\rightarrow R$ are functions, $0$ and $1$ are elements of $R$, $+\subseteq R\times R\times R$ is a 
    relation. We denote $d\in a+b$ for $(a,b,d)\in +$. We require that $(R,+,-,0)$ is a commutative multigroup and that all these 
    satisfying the following  properties for all $a,b,c,d\in 
    R$:
    \begin{description}
        \item [M5 -] $(a\cdot b)\cdot c=a\cdot(b\cdot c)$.
        \item [M6 -] $a\cdot 1=a$.
        \item [M7 -] $a\cdot b=b\cdot a$.
        \item [M8 - ] $a\cdot0=0$.
        \item [M9 - ] If $d\in a+b$ then $cd\in ca+cb$.(weak distributivity) $c.(a+b)  \subseteq  c.a + c.b$
    \end{description}
    
    \begin{ex}\label{ex:1.3}
        $ $
        \begin{enumerate}[a -]
            \item  Every ring, domain and field is gives rise naturally to a {\em strict} multiring, multidomain and multifield, 
            respectively. It  is ease to see that the class of multifields and of hyperfields coincide.
            
            \item (\textup{\cite{marshall2006real}}) $Q_2=\{-1,0,1\}$ is multifield (of signals) with the usual product (in $\mathbb Z$) and the 
            multivalued sum defined by relations
            $$\begin{cases}
                0+x=x+0=x,\,\mbox{for every }x\in Q_2 \\
                1+1=1,\,(-1)+(-1)=-1 \\
                1+(-1)=(-1)+1=\{-1,0,1\}
            \end{cases}
            $$
            This is a hyperfield of characteristic 0 (we will define the characteristic in \ref{char}).
            
            \item (\textup{\cite{jun2015algebraic}}) Let $K=\{0,1\}$ with the usual product and the sum defined by relations $x+0=0+x=x$, $x\in K$ 
            and 
            $1+1=\{0,1\}$. This is a multifield called Krasner's multifield. Obviously, it has characteristic 2.
            
            \item (\textup{\cite{viro2010hyperfields}})In the set $\mathbb R_+$ of positive real numbers, we define 
            $$a\bigtriangledown b=\{c\in\mathbb R_+:|a-b|\le c\le a+b\}.$$ 
            We have $\mathbb R_+$ with the usual product and $\bigtriangledown$ multivalued sum is a multifield, called triangle multifield. We denote 
            this multifield by $\mathcal{T}\mathbb R_+$. Observe that $\mathcal{T}\mathbb R_+$ is not ``double distributive'': 
            $$(2\bigtriangledown1)\cdot(2\bigtriangledown1)=[1,3]\cdot[1,3]=[1,9]$$ 
            and 
            $$2\cdot2\bigtriangledown2\cdot1\bigtriangledown1\cdot2\bigtriangledown1\cdot1=
            4\bigtriangledown2\bigtriangledown2\bigtriangledown1=[0,9].$$
        \end{enumerate}
    \end{ex}
    
    \begin{ex}[Kaleidoscope, Example 2.7 in \textup{\cite{ribeiro2016functorial}}]
        Let $n\in\mathbb{N}$ and define $X_n=\{-n,...,0,...,n\}$. We define the \textbf{\em $n$-kaleidoscope multiring} by 
        $(X_n,+,\cdot,0,1)$, where $+:X_n\times X_n\rightarrow\mathbb{P}(X_n)\setminus\{\emptyset\}$ is given by the rules:
        $$a+b=\begin{cases}
            \{\mbox{sgn}(ab)\max\{|a|,|b|\}\}\mbox{ if }a,b\ne0 \\
            \{a\}\mbox{ if }b=0 \\
            \{b\}\mbox{ if }a=0 \\
            \{-a,...,0,...,a\}\mbox{ if }b=-a
        \end{cases},$$
        and $\cdot:X_n\times X_n\rightarrow\mathbb{P}(X_n)\setminus\{\emptyset\}$ is is given by the rules:
        $$a\cdot b=\begin{cases}
            \mbox{sgn}(ab)\max\{|a|,|b|\}\mbox{ if }a,b\ne0 \\
            0\mbox{ if }a=0\mbox{ or }b=0
        \end{cases}.$$
        
        In this sense, $X_0=\{0\}$ and $X_1=\{-1,0,1\}=Q_2$. 
    \end{ex}
    
    \begin{ex}[H-multifield, Example 2.8 in \textup{\cite{ribeiro2016functorial}}]\label{H-multi}
        Let $p\ge1$ be a prime integer and $H_p:=\{0,1,...,p-1\} \subseteq \mathbb{N}$. Now, define the binary multioperation and operation in $H_p$ as 
        follow:
        \begin{align*}
            a+b&=
            \begin{cases}H_p\mbox{ if }a=b,\,a,b\ne0 \\ \{a,b\} \mbox{ if }a\ne b,\,a,b\ne0 \\ \{a\} \mbox{ if }b=0 \\ \{b\}\mbox{ if }a=0 \end{cases} \\
            a\cdot b&=k\mbox{ where }0\le k<p\mbox{ and }k\equiv ab\mbox{ mod p}.
        \end{align*}
        $(H_p,+,\cdot,-, 0,1)$ is a multifield such that for all $a\in H_p$, $-a=a$. For example, considering $H_3=\{0,1,2\}$, using the above rules we 
        obtain these tables
        \begin{center} 
            \begin{tabular}{|l|l|l|l|l|l|}
                \hline
                $+$ & $0$ & $1$ & $2$\\
                \hline
                $0$ & $\{0\}$ & $\{1\}$ & $\{2\}$ \\
                \hline
                $1$ & $\{1\}$ & $\{0,1,2\}$ & $\{1,2\}$\\
                \hline
                $2$ & $\{2\}$ & $\{1,2\}$ & $\{0,1,2\}$ \\
                \hline
            \end{tabular}
        \end{center}
        
        \begin{center} 
            \begin{tabular}{|l|l|l|l|l|l|}
                \hline
                $\cdot$ & $0$ & $1$ & $2$\\
                \hline
                $0$ & $0$ & $0$ & $0$ \\
                \hline
                $1$ & $0$ & $1$ & $2$\\
                \hline
                $2$ & $0$ & $2$ & $1$ \\
                \hline
            \end{tabular}
        \end{center}
        In fact, these $H_p$ is a kind of generalization of $K$, in the sense that $H_2=K$.
    \end{ex}
    
    Here is a lemma stating the basic properties concerning multirings:
    
    \begin{lem}\label{lemma:1.2}
        Any multiring $R$ satisfies the formulas:
        \begin{enumerate}[a -]
            \item $-(0)=_s0$.
            \item $-(-(x))=_sx$.
            \item $z\sqsubseteq x+y$ $\leftrightarrow$$-(y)\sqsubseteq (-(x))+(-(z))$.
            \item  $-(xy)=_s(-x)y=_sx(-y)$.
        \end{enumerate}
    \end{lem}

    The general definition of the concepts of morphisms and strong morphisms for multialgebraic structures take the following form in 
    the case of multirings:
    
    \begin{defn}[Definition 2.9 in \textup{\cite{ribeiro2016functorial}}]\label{defn:morphism}
        Let $A$ and $B$ multirings. A map $f:A\rightarrow B$ is a morphism if for all $a,b,c\in A$:
        \begin{enumerate}[i -]
            \item $c\in a+b\Rightarrow f(c)\in f(a)+f(b)$;
            \item $f(-a)=-f(a)$;
            \item $f(0)=0$;
            \item $f(ab)=f(a)f(b)$;
            \item $f(1)=1$.
        \end{enumerate}
        
        $f$ is \textbf{a strong morphism} if is a morphism and for all $a,b\in A$, $f(a+b)=f(a)+f(b)$.
    \end{defn}
    
    For multirings, there are types of ``substructure'' that can be considered. Let $f \colon A \to B$ a multiring morphism. 
    If $f$ is injective and a strong morphism, we say that $A$ is \textbf{strongly embedded} in $B$.
    If $f$ is injective, strong morphism and for all $a,b \in A$ and $c \in B$ if $c \in f(a) + f(b)$, then $c \in \mbox{Im}(f)$, 
    then $A$ is a \textbf{submultiring} of $B$.
    Note that in the rings case, all these notions coincide.
    
    To the best of our knowledge, the concept of superring first appears in  (\cite{ameri2019superring}). There are many important advances and results in hyperring theory, and 
    for instance, we recommend for example, the following papers: (\cite{al2019some}), (\cite{ameri2017multiplicative}), (\cite{ameri2019superring}), 
    (\cite{ameri2020advanced}), (\cite{massouros1985theory}), (\cite{nakassis1988recent}), (\cite{massouros1999homomorphic}), (\cite{massouros2009join}).
    
    \begin{defn}[Definition 5 in \textup{\cite{ameri2019superring}}]
        A superring is a structure $(S,+,\cdot, -, 0,1)$ such that:
        \begin{enumerate}[i -]
            \item $(S,+, -, 0)$ is a commutative multigroup.
            
            \item $(S,\cdot,1)$ is a commutative multimonoid. 
            
            \item $0$ is an absorbing element: $a\cdot0= \{0\} = 0 \cdot a$, for all $a\in S$.
            
            \item The weak/semi distributive law holds: 
            if $d\in c.(a+b)$ then $d\in (ca+cb)$,
            for all $a,b,c,d\in S$.  
            
            \item  The rule of signals holds:  $-(ab)=(-a)b=a(-b)$, for all $a,b\in S$.
        \end{enumerate}
        A superdomain is a non-trivial superring without zero-divisors in this new context, i.e. whenever
        $$0\in a\cdot b \mbox{ iff }a=0 \mbox{ or } b=0$$
        A superfield is a non-trivial superring such that every nonzero element is invertible in this new context, i.e. whenever
        $$\mbox{ For all }a \neq 0 \mbox{ exists }b\mbox{ such that }1\in a\cdot b.$$
        A superring is strong if for all $a,b,c,d\in S$, $d\in c\cdot(a+b)$ iff $d\in ca+cb$.
    \end{defn}
    
    \begin{defn}
        Let $A$ and $B$ superrings. A map $f:A\rightarrow B$ is a morphism if for all $a,b,c\in A$:
        \begin{enumerate}[i -]
            \item $c\in a+b\Rightarrow f(c)\in f(a)+f(b)$;
            \item $c\in a\cdot b\Rightarrow f(c)\in f(a)\cdot f(b)$;
            \item $f(-a)=-f(a)$;
            \item $f(0)=0$;
            \item $f(1)=1$.
        \end{enumerate}
        $f$ is \textbf{a strong morphism} if is a morphism and for all $a,b\in A$, $f(a+b)=f(a)+f(b)$ and $f(a\cdot b)=f(a)+f(b)$.
    \end{defn}

    
    The reader interested in Logic but not familiar with multialgebras may have some troubles with the terminology "multi, hyper, super" used in the multialgebra context. For their benefit, we propose the following dictionary that, in particular emphasize the number of multioperations in the structure at sight:
    
    \begin{defn}[Dictionary]
        $ $
        \begin{enumerate}[i -]
            \item ${}_0Ring$ will be denote the (traditional) category of commutative rings with unit; its objects will be called 
            \textbf{0-rings}.
            \item ${}_1Ring$ will be denote the category of commutative multirings; its objects will be called \textbf{1-rings}. 
            ${}_1FRing$ will be denote the category of commutative hyperrings;  Its objects will be called \textbf{full 1-rings}.
            \item ${}_2Ring$ will be denote the category of commutative superrings; its objects will be called \textbf{2-rings}. 
            ${}_2FRing$ will be denote the category of strong commutative superrings; its objects will be called \textbf{full 2-rings}.
        \end{enumerate}
        In this sense, ${}_0Ring={}_0FRing$. These definitions can (and will be) carried to subcategories: for example, ${}_1Field$ is 
        the category of multifields (and we have that ${}_1Field={}_1FField$).
    \end{defn}
    
    From now on, we will use the conventions just above.
    
    Let $(R,+,\cdot, -, 0,1)$ be a 2-ring, $p\in\mathbb N$ and a $p$-tuple $(a_0,a_1, ..., a_{p-1})$.
    
    We define the finite sum  by:
    \begin{align*}
        x\in\sum_{i<0}a_i&\mbox{ iff }x=0, \\
        x\in\sum_{i<p}a_i&\mbox{ iff }x\in y+a_{p-1}\mbox{ for some }y\in\sum_{i<p-1}a_i, \text{if} \ p \geq 1.
    \end{align*}
    
    The finite product is given by:
    \begin{align*}
        x\in\prod_{i<0}a_i&\mbox{ iff }x=1, \\
        x\in\prod_{i<p}a_i&\mbox{ iff }x\in y\cdot a_{p-1}\mbox{ for some }y\in\prod_{i<p-1}a_i, \text{if} \ p \geq 1.
    \end{align*}

    Thus, if $(\vec{a}_0, \vec{a}_1,...,\vec{a}_{p-1})$ is a $p$-tuple of tuples $\vec{a}_i = (a_{i0}, a_{i1},..., a_{i{m_i}})$, then 
    we have the finite sum of finite products:
    \begin{align*}
        x\in\sum_{i<0}\prod_{j<{m_i}}a_{ij}&\mbox{ iff }x=0, \\
        x\in\sum_{i<p}\prod_{j<{m_i}}a_{ij}&\mbox{ iff }x\in y+z\mbox{ for some }y\in
        \sum_{i<{p-1}}\prod_{j<{m_i}}a_{ij} \\
        &\mbox{and } z\in\prod_{j<m_{p-1}}a_{{p-1},j},  \ p \geq 1.
    \end{align*}
    
    Now, we translate some basic facts that holds in rings (0-rings) to 2-rings. Before, we need some terminology:
    
    \begin{defn}\label{char}
        $ $
        \begin{enumerate}[i -]
            \item   An \textbf{ideal} of a 2-ring $A$ is a non-empty subset $\mathfrak{a}$ of $A$ such that 
            $\mathfrak{a}+\mathfrak{a}\subseteq\mathfrak{a}$ and $A\mathfrak{a}\subseteq\mathfrak{a}$. An ideal $\mathfrak{p}$ of $A$ is 
            said to be prime if $1\notin\mathfrak{p}$ and $ab\subseteq\mathfrak{p}\Rightarrow a\in\mathfrak{p}$ or $b\in\mathfrak{p}$. An 
            ideal $\mathfrak{m}$ is maximal if it is proper and for all ideals $\mathfrak{a}$ with 
            $\mathfrak{m}\subseteq\mathfrak{a}\subseteq 
            A$, then $\mathfrak{a}=\mathfrak{m}$ or $\mathfrak{a}=A$. We will denote $\mbox{Spec}(A)=\{\mathfrak{p}\subseteq 
            A:\mathfrak{p}\mbox{ is a prime ideal}\}$.
            
            \item The \textbf{characteristic} of a 2-ring is the smaller integer $n \geq 1$ such that
            $$0\in\sum_{i<n}1,$$
            otherwise the characteristic is zero.
            For full 2-domains, this is equivalent to say that $n$ is the smaller integer such that
            $$\mbox{For all }a,\,0\in\sum_{i<n}a.$$
            
            \item  A \textbf{polynomial expression} in the variables $x_{ij}$, is a multiterm of the form
            $$\sum_{i<p}\prod_{j<{m_i}}x_{ij}.$$
            
            \item  Let $S$ be a subset of a 2-ring $A$. We define the \textbf{ideal generated by} $S$ as $\langle 
            S\rangle:=\bigcap\{\mathfrak{a}\subseteq A\mbox{ ideal}:S\subseteq\mathfrak{a}\}$. If $S=\{a_1,...,a_n\}$, we easily check that
            $$\langle a_1,...,a_n\rangle=\sum Aa_1+...+\sum Aa_n,\,\mbox{where }\sum 
            Aa=\bigcup\limits_{n\ge1}\{\underbrace{Aa+...+Aa}_{n\mbox{ 
                    times}}\}.$$ 
            Note that if $A$ is a full 2-ring, then $\sum Aa=Aa$.
        \end{enumerate}
    \end{defn}
    
    \begin{lem}\label{lem1}
        Let $A$ be a 2-ring.
        \begin{enumerate}[i -]
            \item For all $n\in\mathbb N$ and all $a_0,...,a_{n-1}\in A$, the sum $a_0+...+a_{n-1}$ and product $a_0\cdot...\cdot a_{n-1}$ 
            does not depends on the order of the entries.
            
            \item For every term $t(y_1,...,y_n)$ on the 2-ring language, exists variables $x_{ij}$ such that $A$ satisfies the formula 
            $$t(y_1,...,y_n)\sqsubseteq\sum_{i<p}\prod_{j<m_i}x_{ij}.$$
            Moreover, if $A$ is a {\em full} 2-ring, it satisfies  the formula 
            $$t(y_1,...,y_n) =_s \sum_{i<p}\prod_{j<m_i}x_{ij}.$$
            
        \end{enumerate}
    \end{lem}
    \begin{proof}
        $ $
        \begin{enumerate}[i -]
            \item There is nothing to prove if $n = 0, 1$. If $n=2$ this is just the commutativity of  $+$ and $\cdot$. For $n \geq 3$, the 
            result follows by  induction, using the commutativity and the associativity of $+$ and $\cdot$.
            
            \item This follows by induction on the complexity of the term by the repeated use of associativity of $+$ and $\cdot$ and the 
            weak/semi distributivity law. If $A$ is a full 2-ring, the proof use the full distributivity law instead of its weak version.
            
        \end{enumerate}
    \end{proof}
    
    %
    
    

    \section{Multipolynomials}
    
    This section is devoted to a detailed account of the construction of supperring of polynomials.
    
    Even if the rings-like multi-algebraic structure have been studied for more than 70 years, the idea of considering notions of 
    polynomial in the rings-like multialgebraic structure seems to have considered only in the present century: for instance in 
    \cite{jun2015algebraic} some notion of multi polynomials is introduced to obtain some applications to algebraic and tropical 
    geometry, in \cite{ameri2019superring} a more detailed account of variants of concept of multipolynomials over hyperrings is 
    applied to get a form of Hilbert's Basissatz. 
    
    Our main result in this section is the Theorem \ref{euclid} that provides a Euclidean division algorithm  for 2-rings of 
    multipolynomials in one variable with coefficients in a 1-field.

    Here we will stay close to \cite{ameri2019superring} perspective: let $(R,+,-,\cdot,0,1)$ be a 2-ring and set
    $$R[X]:=\{(a_n)_{n\in\omega}:\exists\,t\,\forall n(n\ge t\rightarrow a_n=0)\}.$$
    Of course, we define the \textbf{degree} of $(a_n)_{n\in\omega}$ to be the smallest $t$ such that $n\ge t\rightarrow a_n=0$. Now 
    define the binary multioperations $+,\cdot :  R[X]\times R[X] \to  \mathcal P^*(R[X])$, a unary operation 
    $-:R[X]\rightarrow R[X]$ and elements $0,1\in R[X]$ by
    \begin{align*}
        (c_n)_{n\in\omega}\in (a_n)_{n\in\omega}+(b_n)_{n\in\omega}&\mbox{ iff }\forall\,n(c_n\in a_n+b_n) \\
        (c_n)_{n\in\omega}\in((a_n)_{n\in\omega}\cdot (b_n)_{n\in\omega}&\mbox{ iff }\forall\,n
        (c_n\in a_0\cdot b_n+a_1\cdot b_{n-1}+...+a_n\cdot b_0) \\
        -(a_n)_{n\in\omega}&=(-a_n)_{n\in\omega} \\
        0&:=(0)_{n\in\omega} \\
        1&:=(1,0,...,0,...)
    \end{align*}
    For convenience, we denote elements of $R[X]$ by $\bm a=(a_n)_{n\in\omega}$. Beside this, we denote
    \begin{align*}
        1&:=(1,0,0,...), \\
        X&:=(0,1,0,...), \\
        X^2&:=(0,0,1,0,...)
    \end{align*}
    etc. In this sense, our ``monomial'' $a_iX^i$ is denoted by $(0,...0,a_i,0,...)$, where $a_i$ is in the $i$-th position; in 
    particular, we will denote ${\underline{b}} = (b,0,0,...)$ and we frequently identify $b \in R \leftrightsquigarrow 
    {\underline{b}} \in R[X]$.
    
    The properties stated  in the lemma below it immediately follows from the definitions involving $R[X]$:
    
    \begin{lem}
        Let $R$ be a 2-ring and $R[X]$ as above and $n,m\in \mathbb N$.
        \begin{enumerate}[a -]
            \item $\{X^{n+m}\}=X^n\cdot X^m$.
            \item For all $a\in R$, $\{aX^n\}= {\underline{a}}\cdot X^n$.
            \item Given $\bm a=(a_0,a_1,...,a_n,0,0,...)\in R[X]$, with with $\deg\bm a \leq n$ and $m\ge1$, we have
            $$\bm aX^m=(0,0,...,0,a_0,a_1,...,a_n,0,0,...)=a_0X^m+a_1X^{m+1}+...+a_nX^{m+n}.$$
            \item For $\bm a=(a_n)_{n\in\omega}\in R[X]$, with $\deg\bm a=t$, 
            $$\{\bm a\}=a_0\cdot1+a_1\cdot X+...+a_t\cdot X^t=a_0+X(a_1+a_2X+...+a_nX^{t-1}).$$
            \item $cX^k.(\bm a + \bm b) = cX^k .\bm a + cX^k.\bm b$.
        \end{enumerate}
    \end{lem}
    
    \begin{fat}
        $ $
        \begin{enumerate}[i -]
            \item $R[X]$ is a 2-ring.
            \item The map $a \in R \mapsto {\underline{a}} = (a,0, \cdots,0, \cdots)$ defines a strong injective 2-ring homomorphism $R 
            \rightarrowtail 
            R[X]$.
            \item For an ordinary ring $R$ (identified with a strict suppering), the 2-ring $R[X]$ is naturally isomorphic to (the 
            2-ring associated to) the ordinary ring of polynomials in one variable over $R$.
        \end{enumerate}
    \end{fat}
    
    \begin{rem} If $R$ is a full 2-ring, does not hold in general that $R[X]$ is also a full 2-ring. In fact, even if $R$ is a 
        1-field, there are examples, e.g. $R = K, Q_2$, such that $R[X]$ is not a full 2-ring (see \textup{\cite{ameri2019superring}}).
    \end{rem}
    
    \begin{defn}
        $R[X]$ will be called the \textbf{2-ring of polynomials} with one variable over $R$. The elements of $R[X]$ will be called 
        (multi)polynomials. We denote $R[X_1,...,X_n]:=(R[X_1,...,X_{n-1}])[X_n]$.
    \end{defn}
    
    \begin{teo} \label{teo Rx}
        $R[X]$ is a 2-domain iff $R$ is a 2-domain. 
    \end{teo}
    \begin{proof}
        We just need to prove that $R[X]$ is a 2-domain iff $R$ is a 2-domain, since the rest is consequence of this.
        ($\Leftarrow$) Let $(a_n)_{n\in\omega},(b_n)_{n\in\omega}\in R[X]$ such that 
        ${\underline{0}}\in(a_n)_{n\in\omega}\cdot(b_n)_{n\in\omega}$.
        Suppose $a_0\ne0$. Since $R$ is a superdomain, we have $b_0=0$. Now, we have $0\in a_0b_1+a_1b_0$, and 
        since $b_0=0$, we conclude $0\in a_0b_1+0$, and so $0\in a_0b_1$. Since $a_0\ne0$, we have $b_1=0$. Repeating this process 
        $t$ steps, when $t$ is the maximum of degrees involved we have that $(b_n)_{n\in\omega}= {\underline{0}}$.
        
        ($\Rightarrow$) Immediate, since $R \rightarrowtail R[x]$ is an injective strong 2-ring homorphism.
    \end{proof}
    
    Now we are read to state and prove the main result in this section.
    
    \begin{teo}[Euclid's Division  Algorithm]\label{euclid}
        Let $K$ be a 1-field. Given polynomials $\bm a,\bm b\in K[X]$ with $\bm b\ne0$, there exists $\bm q,\bm r\in 
        K[X]$ such that $\bm a\in\bm q\bm b+\bm r$, with $\deg \bm r<\deg \bm b$ or $\bm r=0$.
    \end{teo}
    \begin{proof}
        Let $n=\deg\bm a$ and $m=\deg\bm b$. We proceed by induction on $n$. Note that if $m\ge n$, then is sufficient take $\bm q=0$ 
        and $\bm r=\bm a$., so we can suppose $m\le n$. If $m=n=0$, then $\bm a=(a_0,0,...0,...)$ and $\bm b=(b_0,0,...,0,...)$ are both 
        non zero constants, so is sufficient take $\bm q=(a_0/b_0,0,0,...,0,...)$ and $\bm r=0$.
        
        Now, suppose $n\ge1$. Write $\bm a=a_0+X(a_1+...+a_nX^{n-1})=a_0+X\bm a'$ 
        and $\bm b= b_0+X(b_1+...+b_mX^{m-1})=b_0+X\bm b'$, with $a_n,b_m\ne0$ and $\deg\bm a'<n$, $\deg\bm 
        b'<m$.  Then
        \begin{align*}
            \bm a-a_0\cdot1&\in X\bm a' \\
            \bm b-b_0\cdot1&\in X\bm b'.
        \end{align*}
        
        Then $(\bm a-a_0\cdot1)-(\bm b-b_0\cdot1)\subseteq X\bm a'-X\bm b'=X(\bm a'-\bm b')$. Since all polynomials in $\bm 
        a'-\bm b'$ have degree $<n$, by induction we can write
        $$\bm a'-\bm b'\subseteq\bm q\bm b'+\bm r$$
        with $\deg \bm r<\deg\bm b'<m-1$ or $\bm r=0$. Substuting we obtain
        \begin{align*}
            (\bm a-a_0\cdot1)-(b-b_0\cdot1)&\subseteq X(\bm q\bm b'+\bm r)\Rightarrow \\
            (\bm a-a_0\cdot1)-(b-b_0\cdot1)+(b-b_0\cdot1)&\subseteq X(\bm q\bm b'+\bm r)+(b-b_0\cdot1)\Rightarrow \\
            (\bm a-a_0\cdot1)&\subseteq X(\bm q\bm b'+\bm r)+(b-b_0\cdot1)\Rightarrow \\
            \bm a+(a_0\cdot1-a_0\cdot1)&\subseteq X(\bm q\bm b'+\bm r)+(b-b_0\cdot1+a_0\cdot1)\Rightarrow \\
            \bm a&\subseteq X(\bm q\bm b'+\bm r)+(b-b_0\cdot1+a_0\cdot1).
        \end{align*}
        On the other hand,
        \begin{align*}
            X(\bm q\bm b'+\bm r)+(b-b_0\cdot1+a_0\cdot1)&\subseteq
            X(\bm q\bm b'+\bm r)+(b-b_0\cdot1+a_0\cdot1)+\bm q\cdot b_0-\bm q\cdot b_0 \\
            &=\bm q(X\bm b'+b_0)+(X\bm r-\bm q\cdot b_0+a_0\cdot1-b_0\cdot1 ) \\
            &=\bm q\bm b+(X\bm r-\bm q\cdot b_0+a_0\cdot1-b_0\cdot1).
        \end{align*}
        So $\bm a\subseteq \bm q\bm b+(X\bm r-\bm q\cdot b_0+a_0\cdot1-b_0\cdot1)$ with
        $\deg(X\bm r-\bm q\cdot b_0+a_0\cdot1-b_0\cdot1)<\deg\bm b$, as desired.
    \end{proof}
    
    \begin{rem}
        $ $
        \begin{enumerate}[i -]
            \item Note that the polynomials $q$ and $r$ of Theorem \ref{euclid} are not unique in general: if $\bm a\in\bm b\bm q+\bm r$, 
            then $\bm a\in\bm b(\bm q+1-1)+\bm r$ and $\bm a\in\bm b\bm q+(\bm r+1-1)$, then, if $\{0\}\ne1-1$, we have many $q$'s and $r$'s.
            
            However, if $R$ is a ring (or 0-ring, in agreement with our notation), then Theorem \ref{euclid} provide the usual Euclid 
            Algorithm, with the uniqueness of the quotient and remainder.
            
            \item The Theorem \ref{euclid} above gives immediately another proof of Theorem 6 in \cite{ameri2019superring}. i.e. every ideal 
            in $K[X]$ is a principal ideal: for a non zero ideal $I \subseteq K[X]$ select a nonzero polynomial $\bm b(x) \in I$ with minimal 
            degree, then $I = F[X]. \bm b(x)$.
            
        \end{enumerate}
        
    \end{rem}
    
    \section{Beginning the model theory of algebraically closed multifields}
    
    This section contains the main contributions of this paper: we  introduce some concept of  the algebraically closed 1-field 
    and give the first steps on a model theory of this class with a kind of quantifier elimination procedure.
    
    \subsection{On algebraically closed multifields}
    
    Let $R, S$ be  2-rings and $h : R \to S$ be a morphism. Then $h$ extends naturally to the 2-rings multipolynomials $h^X : 
    R[X] \to S[X]$:
    $$(a_n)_{n \in \mathbb{N}} \in R[X] \ \mapsto \ (h(a_n))_{n \in \mathbb{N}} \in S[X]$$
    
    Now let $s \in S$  we have the $h$-\textbf{evaluation} of $s$ at $\bm a\in R[X]$, $degree(\bm a) \leq n$ by
    $$\bm a^h(s)=ev^h(s,\bm a)=\{s'\in S : s'\in h(a_0)+ h(a_1).s+h(a_2).s^2+...+h(a_n).s^n\}.$$
    
    In particular if $T\supseteq R$ is a 2-ring extension and $\alpha\in T$, we have the \textbf{evaluation} of $\alpha$ at $\bm a\in 
    R[X]$ by
    $$\bm a(\alpha)=ev(\alpha,\bm a)=\{b\in T: b \in a_0+a_1\alpha+a_2\alpha^2+...+a_n\alpha^n\}.$$
    
    A \textbf{root} of $\bm a$ in $T$ is an element $\alpha\in T$ such that $0\in ev(\alpha,\bm a)$. A 2-ring $R$ is 
    \textbf{algebraically closed} if every non constant polynomial in $R[X]$ has a root in $R$.
    
    Observe that, if $F$ is a field, the evaluation of $F[X]$ as a 1-ring coincide with the usual evaluation. Therefore, if $F$ is 
    algebraically closed as 1-field and 2-field, then will be algebraically closed in the usual sense.
    
    
    Unfortunately, in dealing with multipolynomials, strange situations appears:
    
    \begin{ex}[Finite Algebraically Closed 1-Field]\label{exluc}
        The 1-field $K=\{0,1\}$ is algebraically closed. In fact, if $\bm p=a_0+a_1X+a_2X^2+...+a_nX^n\in K[X]$, with $a_n\ne0$, then 
        $\bm p(1)=K$, since $1+1=\{0,1\}$.
    \end{ex}
    
    
    \subsection{A quantifier elimination procedure}
    
    Instead of these "anomalies", we have a quantifier elimination procedure for any {\em infinite} algebraically closed 1-fields. We 
    will 
    describe this (that is a variation of Theorem 9.2.1 in \cite{jarden2008field}) after the following technical lemma:
    
    \begin{lem}[Reduction Lemma]\label{reduc}
        Let $A$ be a 2-ring, $t_1(\bar{x}),t_2(\bar{x})$ be terms on the full 2-ring language  and let $v = \bar{a} : \bar{x} \to A$
        \begin{enumerate}[i -]
            \item $ t_1^A(\bar{a})\subseteq t_2^A(\bar{a})$ iff $0\in (t_2-t_1)^A(\bar{a})$.
            \item Given any atomic formula,  $t_1(\bar{x}) \sqsubseteq t_2(\bar{x})$, there is a polynomial term $p(\bar{x}) \in R[\bar{x}]$ 
            such that
            $$A\models_v (t_1(\bar{x})\sqsubseteq t_2(\bar{x})) \ \leftrightarrow \ (0\sqsubseteq p(\bar{x})).$$
        \end{enumerate}
    \end{lem}
    \begin{proof}
        Item (i) follows immediately from the axiom {\bf M1} of superrings. Item (ii) follows from item (i) above, the item (ii) of Lemma 
        \ref{lem1} and by a repeated use of Theorem \ref{teo Rx}.
    \end{proof}
    
    Let $\mathcal L$ be the language of 1-rings. For each 1-ring $R$, let $\mathcal L(R)$ be the language extending $\mathcal L$ 
    by adding all elements of $R$ as {\em strict} constant symbols. Let $\Gamma'$ be the 1-ring axioms. Let extend $\Gamma'$ by 
    (in)equalities and relations of the form
    $$a_0 \neq b_0; \  c_1 = a_1.b_1; \ c_2\in a_2+b_2; \,a_i,b_i,c_i\in R$$
    that are true in $R$ ("the diagram of $R$"). Denote the set of formulas obtained by $\Gamma'(R)$. A model of $\Gamma'(R)$ is a 
    1-ring that contains a subset $\overline 
    R=\{\overline a:a\in R\}$ and $\overline R$ is an isomomorphic copy of $R$ inside this model.
    
    If $R=K$ is a 1-field and  $\Gamma$ is the 1-field axioms, then a model of $\Gamma(K)$ is a 1-field that contains a subset 
    $\overline 
    K=\{\overline a:a\in K\}$ and $\overline K$ is a 1-field isomorphic to $K$. Then a model of $\Gamma(K)$ is (up to a isomorphism) a 
    1-field 
    containing $K$.
    
    Now, we extend $\Gamma(K)$ to a new set of axioms $\tilde\Gamma(K)$ adding
    \begin{align*}
        \tag{AC}\forall\,z_0...\forall\,z_n\,\exists\,x[0\in z_0+z_1x+...+z_{n-1}x^{n-1}+x^n],\ n\ge1.
    \end{align*}
    and because the counter Example \ref{exluc}, we add also the family of axioms
    
    $$\exists z_0 ...\exists z_{n-1} \bigvee\limits_{i<j<n}[z_i\ne z_j], \ n\geq 2.$$

    A model $F$ of $\Gamma(K)$ is also a model of $\tilde\Gamma(K)$ iff $F$ is infinite and algebraically closed. Our aim is to 
    describe a 
    quantifier elimination procedure for $\tilde\Gamma(F)$. By the reduction Lemma \ref{reduc}, $F$ regards every atomic formula as 
    equivalent modulo 
    $\Gamma(K)$ to a polynomial ``equation'' $0\in f(X_1,...,X_n)$.
    
    Since $K[\bar{X}]$ is a  2-domain (by an iteration of Theorem \ref{teo Rx}), a conjunction of inequations 
    $$\bigwedge\limits^m_{i=1}[0\neq g_i(\bar{X})]$$
    is equivalent to the ``inequation" $0\notin g_1(\bar{X})...g_n(\bar{ X})$. Then, to obtain a quantifier elimination for 
    $\tilde\Gamma(K)$ is 
    sufficient eliminate $Y$ from the formula
    \begin{align}\label{qe1}
        \exists\,Y[0\in f_1(\bar{ X},Y)\wedge...\wedge0\in f_m(\bar{ X},Y)\wedge 0\notin g(\bar{X},Y)]
    \end{align}
    with $f_1,...,f_m,g\in R[X_1,...,X_m,Y]$.
    
    \begin{teo}[Quantifier Elimination Procedure]\label{quantfield}
        Let $K$ be an infinite  1-field and $\varphi(X_1,...,X_n,Y)$ the formula in \ref{qe1}. Then $\varphi(X_1,...,X_n,Y)$ is 
        equivalent modulo $\tilde\Gamma(R)$ to a boolean combination of  atomic formulas 
        $\psi(X_1,...,X_r)$, $r\ge n$.
    \end{teo}
    \begin{proof}
        The proof consists in three parts:
        
        $ $
        
        {\bf A -} Reduction to the case that only one of $f_1,...,f_m$ involves $Y$.
        
        Move each conjunction that appears in (\ref{qe1}) and that does not involve $Y$ to the left of $\exists\,Y$ according to the rule 
        ``$\exists Y[\varphi\wedge\psi]\equiv\varphi\wedge\exists Y[\psi]$ if $Y$ does not appear in $\varphi$''. Thus we assume 
        $\deg_Y(f_i(\bar{X},Y))\ge1$, $i=1,...,m$ and $m\ge2$.
        
        We now perform an induction on $\sum\deg_Y(f_i(\bar{X},Y))$:
        
        Let $p(\bar{X},Y)$ and $q(\bar{X},Y)$ be multipolynomials with coefficients in $R$ such that $0\le\deg_Y p(\bar{ 
            X},Y)\le\deg_Yq(\bar{
            X},Y)=d$. Write $p(\bar{X},Y)$ in the form
        \begin{align}\label{qe2}
            p(\bar{X},Y)=a_k(\bar{ X})Y^k+a_{k-1}(\bar{ X})Y^{k-1}+...+a_0(\bar{X})
        \end{align}
        with $a_j\in R[\bar{X}]$. For each $j$ with $0\le j\le k$ let 
        \begin{align*}
            p_j(\bar{X},Y)=a_j(\bar{ X})Y^j+a_{j-1}(\bar{ X})Y^{j-1}+...+a_0(\bar{X})
        \end{align*} 
        If $0\notin a_j(\bar{X})$, division of $q(\bar{X},Y)$ by $p_j(\bar{X},Y)$ produces $q_j(\bar{X},Y)$ and $r_j(\bar{X},Y)$ in 
        $R[\bar{X},Y]$ 
        for which
        \begin{align}\label{qe3}
            a_j(\bar{X})^dq(\bar{X},Y)\subseteq q_j(\bar{X},Y)p_j(\bar{X},Y)+r_j(\bar{X},Y),
        \end{align}
        and $\deg_Y(r_j)<\deg_Y(p_j)\le d$.
        
        Let $F$ be a model of $\Gamma(K)$. If $x_1,...,x_n,y$ are elements  of $F$ such that $0\in a_l(\bar{x})$ for $l=j+1,...,k$ and 
        $0\notin a_j(\bar{x})$, then $[0\in p(\bar{x},y)\wedge0\in q(\bar{x},y)]$ is equivalent in $F$ to $[0\in p_j(\bar{x},y)\wedge0\in 
        r_j(\bar{x},y)]$. Therefore, the formula $[0\in p(\bar{X},Y)\wedge0\in q(\bar{X},Y)]$ is equivalent modulo $\Gamma(K)$ to the 
        formula
        \begin{align}\label{qe4}
            \left(\bigvee\limits^k_{j=0}[0\in a_k(\bar{X})\wedge...\wedge0\in a_{j+1}(\bar{X})\wedge0\notin a_j(\bar{X})\wedge
            0\in p_j(\bar{X},Y)\wedge0\in r_j(\bar{X},Y)]\right) \nonumber \\
            \vee[0\in a_k(\bar{X})\wedge...\wedge0\in a_0(\bar{X})\wedge0\in q(\bar{X},Y)].
        \end{align}
        Apply the outcome of (\ref{qe4}) to $f_1(\bar{X},Y)$ and $f_m(\bar{X},Y)$ (of \ref{qe1}). With the rule ``
        $\exists\,Y[\varphi\vee\psi]\equiv\exists\,Y\varphi\vee\exists\,Y\psi$'' we have replaced (\ref{qe1}) by disjunction of statements 
        of form (\ref{qe1}) in each which the sum corresponding to $\sum\deg_Y(f_i(\bar{X},Y))$ is smaller. Using the induction assumption, 
        we conclude that $m$ may be taken to be at most $1$.
        
        $ $
        
        {\bf B -} Reduction to the case that $m=0$.
        
        Continue the notation of part $A$ which left us at the point  of considering how to eliminate $Y$ from $p(\bar{X},Y)$ in
        \begin{align}\label{qe5}
            \exists\,Y[0\in p(\bar{X},Y)\wedge0\notin g(\bar{X},Y)].
        \end{align}
        Consider a model $F$ of $\tilde\Gamma(K)$ and elements $x_1,...,x_n\in F$. If $0\notin p(\bar{x},Y)$ then (since $
        F$ is algebraically closed) the statement 
        $$F\models\exists\,Y[0\in p(\bar{x},Y)\wedge0\notin g(\bar{x},Y)]$$
        is equivalent to the statement
        $$p(\bar{x},Y)\mbox{ does not divide }g(\bar{x},Y)^k\mbox{ in } F[X].$$
        Therefore, with $q(\bar{X},Y)=g(\bar{X},Y)^k$ and in the notation of (\ref{qe2}) and (\ref{qe3}), formula (\ref{qe5}) is equivalent 
        modulo 
        $\tilde\Gamma(K)$ to the formula
        \begin{align*}
            \left(\bigvee\limits^k_{j=0}[0\in a_k(\bar{X})\wedge...\wedge0\in a_{j+1}(\bar{X})\wedge0\notin a_j(\bar{X})\wedge
            \exists Y[\in r_j(\bar{X},Y)]]\right) \nonumber \\
            \vee[0\in a_k(\bar{X})\wedge...\wedge0\in a_0(\bar{X})\wedge\exists Y[0\in g(\bar{X},Y)]]
        \end{align*}
        a disjunction of statements of form (\ref{qe1}) with $m=0$.
        
        $ $
        
        {\bf C -} Completion of the proof.
        
        By part B we are in the point of removing $Y$ from a statement of the form
        $$\exists\,Y[0\notin a_l(\bar{X})Y^l+a_{l-1}(\bar{X})Y^{l-1}+...+a_0(\bar{X})].$$
        Since models of $\tilde\Gamma(K)$ are infinite 1-fields, this formula is equivalent modulo $\tilde\Gamma(K)$ to
        $$0\notin a_l(\bar{X})\vee...\vee0\notin a_0(\bar{X}),$$
        completing the quantifier elimination procedure.
    \end{proof}
    
    \begin{rem}
        $ $
        \begin{enumerate}[i-]
            \item The previous result subsumes the usual one, i.e., if $K$ is a ordinary algebraically closed field then it is an infinite 
            algebraically closed (strict) 1-field and Theorem \ref{quantfield} is just the usual quantifier elimination result.
            
            \item In general, when the translate the Theorem \ref{quantfield} to a result on first-order relational structures (see Remark 
            \ref{translation-rem}) we get that the theory of algebraically closed 1-fields perceives that some formulas are always equivalent 
            to some  $\forall \exists$-formulas: but this results seems not so ease to have a previous and direct intuition (and consequent 
            proof), i.e.  without the use of the language of multialgebras and its results.
        \end{enumerate}
    \end{rem}
    
    \section{Final remarks and future works}
    
    We finish the work presenting here some possible future developments.
    
    \begin{itemize}
        \item It could be interesting describe and explore an alternative notion of algebraically closed multifield based on an 
        alternative notion of   of root of a polynomial, taking in account factorizations, for example, if $p(x)\in(x-b)q(x)$ for some 
        $q(x)$, then $b$ can be seem as a root of $p(x)$: by Theorem 7 in \cite{ameri2019superring}, this in fact {\em coincide} with the 
        other notion of root of a  polynomial $p(x) \in F[x]$ whenever $F$ is a hyperfield.
        \item It is true that any full 1-field has a kind of algebraic closure?
        \item In what sense the theory of algebraically closed 1-fields could be considered a model completion of the 1-field theory?
        \item We gave a first step in model theory of 1-fields. It will be interesting consider,  in the same line, the model theory of 
        1-fields endowed with some extra structure: orderings (\cite{marshall2006real}, valuations (\cite{jun2018valuations}, etc.
        \item The Theorem \ref{quantfield} could be adapted for real closed and henselian 1-fields?
        \item In \cite{pelea2006multialgebras}  was started the development of a identity theory and a universal algebra like theory 
        for multi structures. However,  a full model theory of multi structures, in the vein of chapter 1 of 
        \cite{diaconescu2008institution}, should be an object of interest (as the present work suggests) and it is seems to be unknown.
    \end{itemize}
    
    \bibliographystyle{plain} 
    \bibliography{one_for_almost_all.bib} 
    
\end{document}